\documentclass[a4paper,12pt]{article}

\newenvironment{proof}{\noindent {\bf Proof:}}{\hfill $\Box$}

\newtheorem{lemma}{Lemma}

\newtheorem{remark}{Remark}

\usepackage{amsmath}
\usepackage{amssymb}
\usepackage{amsfonts}
\usepackage{graphicx}

\title{\bf 
Moment LMI approach\\
to LTV impulsive control
}

\begin{document}

\author{Mathieu Claeys$^{1,2}$, Denis Arzelier$^{1,2}$,\\
Didier Henrion$^{1,2,3}$, Jean-Bernard Lasserre$^{1,2,4}$}

\footnotetext[1]{CNRS; LAAS; 7 avenue du colonel Roche, F-31077 Toulouse; France.}
\footnotetext[2]{Universit\'e de Toulouse; UPS, INSA, INP, ISAE; UT1, UTM, LAAS; F-31077 Toulouse; France}
\footnotetext[3]{Faculty of Electrical Engineering, Czech Technical University in Prague,
Technick\'a 2, CZ-16626 Prague, Czech Republic}
\footnotetext[4]{Institut de Math\'ematiques de Toulouse, Universit\'e de
Toulouse; UPS; F-31062 Toulouse, France.}

\date{Draft of \today}

\maketitle

\begin{abstract}
In the 1960s, a moment approach to linear time varying (LTV) minimal norm impulsive optimal control
was developed, as an alternative to direct approaches (based on discretization of the equations
of motion and linear programming)
or indirect approaches (based on Pontryagin's maximum principle). This paper revisits these
classical results in the light of recent advances in convex optimization, in particular the use
of measures jointly with hierarchy of linear matrix inequality (LMI) relaxations.
Linearity of the dynamics allows us to integrate system trajectories and
to come up with a simplified LMI hierarchy where the only unknowns are moments of a
vector of control measures of time. In particular, occupation measures of state and
control variables do not appear in this formulation. This is in stark contrast with
LMI relaxations arising usually in polynomial optimal
control, where size grows quickly as a function of the relaxation order. Jointly with
the use of Chebyshev polynomials (as a numerically more stable polynomial basis),
this allows LMI relaxations of high order (up to a few hundreds) to be solved numerically.
\end{abstract}

\section{Introduction}

In the 1960s, it was realized that many physically relevant problems
of optimal control were inappropriately formulated in the sense
that the optimum control law (a function of time and/or state)
cannot be found if the admissible functional space is too small.
This observation was the main driving force of the papers
\cite{neustadt,rishel,schmaedeke} which introduce optimal control
problems formulated in the space of measures. The approach
is well summarized in the textbook
\cite{luenberger}, which contains some (academic) examples
of optimal control problems without solutions. This motivated the
introduction of many concepts of functional analysis (density,
completeness, duality, separability) in control engineering,
building up on the advances on mathematical control theory
and calculus of variations. As promoted in \cite{luenberger},
an optimal control problem should be formulated in the dual
of a Banach space which is large enough
for the solution to be attained.

Most of the literature on numerical optimal control focuses
on direct approaches (based on discretization of the equations
of motions and linear programming) or indirect approaches
(based on the necessary optimal conditions of Pontryagin's
maximum principle). When applied to optimal control problems
whose optima cannot be attained, these numerical approaches
typically face difficulties. In this context, we believe
that it is timely to revisit classical results by Neustadt \cite{neustadt}
on the formulation of optimal control problems for
linear time varying (LTV) systems as a problem of
moments, where the decision variables (from which an optimal
control law can be extracted) are measures subject to
a finite number of linear constraints.

There is an important literature, especially from the 1960s,
on moment formulations to optimal control of ordinary differential
equations (ODEs) and partial differential equations (PDEs),
see e.g. \cite{avdonin} and references therein, as well as the comments of
\cite[p. 586]{fattorini}. Currently,
this approach is not frequently used by engineers,
and in our opinion this may be due, on the one hand,
to the technicality of the underlying concepts of functional analysis,
and, on the other hand, to
the absence of numerical methods to deal satisfactorily
with optimization problems in large functional spaces
such as spaces of measures or distributions.
Regarding the first point, we strongly recommend
the textbook \cite{luenberger} which is a very readable
account of elementary functional analysis useful for
engineers. Regarding the second point, there has been
recent advances in convex optimization, especially
semidefinite programming (optimization over linear
matrix inequalities, LMIs), for solving numerically
generalized problems of moments, i.e. linear programming
(LP) problems in Banach spaces of measures, see
\cite{lasserre,gloptipoly3} and references therein.

Our contribution is therefore to revisit the classical
formulation by Neustadt \cite{neustadt} in the light of
recent advances on LMI hierarchies for solving
generalized problems of moments. In our previous work
\cite{sicon}, we formulated polynomial optimal control
problems with semialgebraic state and control constraints
as generalized problems of moments that can be
solved with asymptotically converging LMI hierarchies.
The optimal control problem is relaxed to an LP problem
in the space of occupation measures, which are
measures of time, state and control encoding
the trajectories of the system.
The main drawback of this approach is the rapid growth
of the size of the LMI problems in the hierarchy, making
the approach applicable to small-size problems only
(say at most 3 states and 2 controls). In the
current paper, which focuses on the specific case
of LTV dynamics depending affinely on the control variable,
we first replace control variables by interpreting
them as measures of time. This is similar to our previous work \cite{impulse}
which also dealt with impulsive control design in a more general setting. Second, we
get rid of the state variables by integrating numerically
the LTV ODE. This is possible because the ODE depends
linearly on the state and the control. As a result, the optimal control problem
is relaxed to an LP on measures depending only
on time. It follows that there is no need for an LMI
hierarchy since in the univariate case finite-dimensional
moment LMI conditions are necessary and sufficient.
However, there is still a hierarchy of LMI conditions
in connection with the polynomial approximations
of increasing degree we use to model the integrated system
trajectories. To deal with those high degree univariate polynomials
and moment matrices, we use Chebyshev polynomials instead
of monomials. Indeed, high degree Chebyshev polynomials behave much better
numerically than monomials, and we can rely on functionalities
of the {\tt chebfun} package for Matlab \cite{chebfun,trefethen}
to integrate the LTV ODE and manipulate polynomials. To illustrate the above methodology, we show on some examples how to obtain very good approximations of impulse times and amplitudes of an optimal solution.

\section{Relaxed linear optimal control}

Consider the linear time varying (LTV) optimal control problem
\begin{equation}\label{ocp}
\begin{array}{rcll}
q^* & := & \inf & \displaystyle \|u\|_1 := \sum_{j=1}^m \int_{t_I}^{t_F} |u_j(t)|dt \\
& & \mathrm{s.t.} & \dot{x}(t) = A(t)x(t) + B(t)u(t) \\
& & & x(t_I) \in {\mathbb R}^n \quad\mathrm{given}\\
& & & x(t_F) \in {\mathbb R}^n \quad\mathrm{given}
\end{array}
\end{equation}
where the minimization is w.r.t. a vector of control functions
$u_j \in L^1([t_I,t_F])$, $j=1,\ldots,m$, and
$A \in L^{\infty}([t_I,t_F]; {\mathbb R}^{n\times n})$,
$B \in C([t_I,t_F]; {\mathbb R}^{n\times m})$,
on a given bounded time interval $[t_I,t_F] \subset {\mathbb R}$.

In general the infimum is not attained, and the optimal control problem is relaxed to
\begin{equation}\label{p}
\begin{array}{rcll}
p^* & := & \inf & \displaystyle \|\mu\|_{TV} := \sum_{j=1}^m \int_{t_I}^{t_F} |\mu_j|(dt) \\
& & \mathrm{s.t.} & x(dt) = A(t)x(t)dt + B(t)\mu(dt) \\
& & & x(t_I) \in {\mathbb R}^n \quad\mathrm{given}\\
& & & x(t_F) \in {\mathbb R}^n \quad\mathrm{given}
\end{array}
\end{equation}
where the minimization is w.r.t. a vector of (signed) measures
$\mu_j \in M([t_I,t_F])$, $j=1,\ldots,m$, and $\|\mu\|_{TV}$ denotes
the total variation, or norm, of vector measure $\mu$.
A measure in $M([t_I,t_F])$ of finite norm
is identified (by a representation theorem
of F. Riesz, see e.g. \cite[Section 21.5]{royden})
as a continuous linear functional acting
on the space of continuous functions $C([t_I,t_F])$.

Problem (\ref{p}) is a relaxation of problem (\ref{ocp})
since we enlarge the space of admissible controls.
Indeed, problem (\ref{ocp}) is equivalent to problem (\ref{p}) restricted
to measures which are absolutely continuous w.r.t. time, i.e. $\mu_j(dt) = u_j(t)dt$
for some $u_j \in L^1([t_I,t_F])$, $j=1,\ldots,m$.
The motivation for introducing relaxed problem (\ref{p}) is as follows.

\begin{lemma}\label{relax}
The infimum is attained in problem (\ref{p}) and it is equal to the infimum of problem (\ref{ocp}),
i.e. $q^*=p^*$.
\end{lemma}

The proof of Lemma \ref{relax} is relegated to the end of Section \ref{moments}.
We will introduce a numerical method to deal directly
with relaxed problem (\ref{p}) in measure space $M$,
bypassing the potential difficulties
coming from the fact that the infimum in problem (\ref{ocp})
is typically not attained in function space $L^1$.
Before this, we need to reformulate optimal control problem (\ref{p}) as a
problem of moments.

\section{Problem of moments}\label{moments}

Now we integrate the differential equation in problem (\ref{p}) to obtain
an equivalent problem of moments. For more details, see e.g. \cite[Section 2.2]{bressan}.

Let $f_i \in W^{1,1}([t_I,t_F];{\mathbb R}^n)$ denote the absolutely continuous
solution of the Cauchy problem $\dot{f_i}(t) = A(t)f_i(t)$ with $f_i(t_I)$
equal to the $i$-th column of $I_n$, the $n$-by-$n$ identity matrix, for $i=1,\ldots,n$.
The matrix
\[
F(t) := \left[\begin{array}{ccc}
f_1(t) & \cdots & f_n(t)
\end{array}\right] \in W^{1,1}([t_I,t_F]; {\mathbb R}^{n\times n})
\]
therefore satisfies\footnote{A function belongs to the Sobolev space $W^{1,1}$
of absolutely continuous functions if it is the integral of a function
of Lebesgue space $L^1$.} the matrix ODE
\[
\dot{F}(t)=A(t)F(t), \quad F(t_I) = I_n, \quad t \in [t_I,t_F].
\]
From \cite[Theorem 2.2.3]{bressan}\footnote{In this reference the authors
define the fundamental matrix solution as a bivariate matrix $M(t,s)$
such that $\partial M(t,s)/\partial t = A(t)M(t,s)$, $M(s,s)=I_n$.
The connection with our definition is that $F(t)=M(t,t_I)$.},
matrix $F^{-1}(t)$ is differentiable
and any function $x(t)$ satisfying
\[
x(t_F) = F(t_F)\left[x(t_I) + \int_{t_I}^{t_F} F^{-1}(t)B(t)\mu(dt)\right]
\]
is a solution to the differential equation
\begin{equation}\label{ode}
x(dt) = A(t)x(t)dt+B(t)\mu(dt), \quad t \in [t_I,t_F].
\end{equation}
Letting
\[
\begin{array}{rcl}
G(t) & := & (F^{-1}(t) B(t))^T \\
& = & \left[\begin{array}{ccc}
g_1(t) & \cdots & g_n(t)
\end{array}\right]
\in C([t_I,t_F]; {\mathbb R}^{m\times n}), \\[1em]
h & := & F^{-1}(t_F)x(t_F) - x(t_I) \in {\mathbb R}^m,
\end{array}
\]
we can replace the differential equation (\ref{ode})
with the integral equation:
\[
\int_{t_I}^{t_F} G^T(t)\mu(dt) = h.
\]
It follows that problem (\ref{p}) can be written equivalently as
\begin{equation}\label{m}
\begin{array}{rcll}
p^* & = & \min & \displaystyle \|\mu\|_{TV} \\
& & \mathrm{s.t.} & \displaystyle \langle g_i, \mu \rangle = h_i, \quad i=1,\ldots,n
\end{array}
\end{equation}
where
\[
\langle g_i, \mu \rangle := \int_{t_I}^{t_F} g^T_i(t) \mu(dt) =
\sum_{j=1}^m \int_{t_I}^{t_F} g_{i,j}(t) \mu_j(dt)
\]
denotes the duality bracket between $C([t_I,t_F]; {\mathbb R}^m)$ and
its dual $M([t_I,t_F]; {\mathbb R}^m)$, a bilinear form pairing $C([t_I,t_F]; {\mathbb R}^m)$ and $M([t_I,t_F]; {\mathbb R}^m)$.
Problem (\ref{m}) is a problem of moments consisting of finding $m$ measures
subject to $n$ linear constraints.

\begin{proof} {\sl (of Lemma \ref{relax})}
The proof that the infimum of problem (\ref{p}) is attained follows
from \cite[Theorem 1]{neustadt}. Moreover, in \cite[Theorem 4]{neustadt} it is
shown that there exists a solution $\mu$ to problem (\ref{m}), hence
to problem (\ref{p}),
which is the (signed) sum of at most $n$ Dirac measures.
Finally, it is shown in \cite[pp. 45-46]{neustadt} that there is
a sequence of functions $u^k \in L^1([t_I,t_F]; {\mathbb R}^m)$,
$k=1,2,\ldots$ which are admissible for problem (\ref{ocp}),
i.e. $\mu(dt)=u^k(t)dt$ satisfies the constraints in problem (\ref{m}),
and which are such that $\lim_{k\to\infty}\|u^k\|_1 = p^*$.
\end{proof}

\section{Primal and dual conic LP}\label{conic}
By decomposing each signed measure
$\mu_j$ as a difference of two nonnegative measures (using
the Jordan decomposition theorem, see e.g. \cite[Section 17.2]{royden}), i.e.
\[
\mu_j = \mu^+_j - \mu^-_j, \quad \mu^+_j \geq 0, \quad \mu^-_j \geq 0, \quad j=1,\ldots,m,
\]
problem (\ref{m}) can be written as a linear programming (LP) problem
on the cone of nonnegative measures
\begin{equation}\label{plp}
\begin{array}{rcll}
p^* & = & \inf & \displaystyle  \langle 1,\mu^+ \rangle + \langle 1, \mu^- \rangle \\
& & \mathrm{s.t.} & \displaystyle \langle g_i, \mu^+ \rangle - \langle g_i, \mu^-
\rangle = h_i, \quad i=1,\ldots,n \\
& & & \mu^+ \geq 0, \quad \mu^- \geq 0
\end{array}
\end{equation}
where $1$ denotes the $m$-dimensional vector of functions identically equal to one,
and the above minimization is w.r.t. two vector-valued nonnegative measures
$\mu^+ \in M([t_I,t_F]; {\mathbb R}^m)$, $\mu^- \in M([t_I,t_F]; {\mathbb R}^m)$.
It is easy to show that problem (\ref{plp}) is equivalent to problem (\ref{m}).


Problem (\ref{plp}) is the dual of the following LP on the cone of nonnegative continuous functions (see \cite{Shapiro2001} for details of the derivation):
\begin{equation}\label{dlp}
\begin{array}{rcll}
d^* & = & \sup & \displaystyle \sum_{i=1}^n y_i h_i \\
& & \mathrm{s.t.} & \displaystyle z^+(t): = 1 + \sum_{i=1}^n y_i g_i(t) \geq 0
\quad t \in[t_I,t_F]\\
& & & z^-(t): = \displaystyle 1 - \sum_{i=1}^n y_i g_i(t) \geq 0
\quad t \in[t_I,t_F]
\end{array}
\end{equation}
where the maximization is w.r.t. a vector $y \in {\mathbb R}^n$ parametrizing
two vector-valued nonnegative continuous functions
$z^+ \in C([t_I,t_F]; {\mathbb R}^m)$, $z^- \in C([t_I,t_F]; {\mathbb R}^m)$.
Denoting
\[
\|z\|_{\infty}:=\sup_{t\in[t_I,t_F], j=1,\ldots,m}|z_j(t)|
\]
for any $z \in C([t_I,t_F]; {\mathbb R}^m)$,
remark that LP problem (\ref{dlp}) can be also written as
\begin{equation}
\begin{array}{rcll}
d^* & = & \sup & h^T y \\
& & \mathrm{s.t.} & \|G(t) y\|_{\infty} \leq 1.
\end{array}
\label{formulation-classique}
\end{equation}

\begin{lemma}\label{nogap}
There is no duality gap between LP (\ref{plp}) and (\ref{dlp}), i.e. $p^*=d^*$.
\end{lemma}

\begin{proof}
Define the vector $r(\mu^+,\mu^-) \in {\mathbb R}^{n+1}$ with entries
$r_0(\mu^+,\mu^-):=\langle 1,\mu^+\rangle+\langle 1,\mu^-\rangle$,
$r_i(\mu^+,\mu^-):=\langle g_i,\mu^+\rangle-\langle g_i,\mu^-\rangle$,
$i=1,\ldots,n$ and the set
$R:=\{r(\mu^+,\mu^-) \: :\: (\mu^+,\mu^-) \in M^{2m}_+\} \subset {\mathbb R}^{n+1}$
where $M^{2m}_+$ denotes the cone of nonnegative measures in $M([t_I,t_F];{\mathbb R}^{2m})$.
Let us invoke \cite[Theorem 3.10]{anderson} which states that $p^*=d^*$ provided $p^*$
is finite and $R$ is closed. Finiteness of $p^*$ follows immediately since we minimize
the total variation. To prove closedness, we have to show that all accumulation points
of any sequence $r(\mu^+_n,\mu^-_n)$ belong to $R$. Since the supports
of the measures are compact and $p^*$ is finite, hence $\mu^+$ and $\mu^-$ are bounded,
the sequence $(\mu^+_n,\mu^-_n)$ is bounded. By the weak-* compactness of the
unit ball in the Banach space of bounded signed measures with compact support
(Alaoglu's Theorem, see e.g. \cite[Section 5.10]{luenberger}
or \cite[Section 15.1]{royden}), there is a subsequence $(\mu^+_{n_k},\mu^-_{n_k})$
that converges weakly-* to an element $(\mu^+,\mu^-) \in M^{2m}_+$. As $1$ and all $g_i$ belong to $C([t_I,t_F]; {\mathbb R}^m)$ then
$\lim_{k\to\infty} r(\mu^+_{n_k},\mu^-_{n_k})=r(\mu^+,\mu^-) \in R$. 
\end{proof}

Zero duality gap implies that any optimal pair $((\mu^+,\mu^-),(z^+,z^-))$
solving LPs (\ref{plp}-\ref{dlp}) satisfies the complementarity conditions
\[
\langle z^+_j,\mu^+_j \rangle = 0, \quad \langle z^+_j,\mu^-_j \rangle = 0, \quad j=1,\ldots,m.
\]
This means that the support of each measure $\mu^+_j$, resp. $\mu^-_j$, is included in the set
$\{t \in [t_I,t_F] : z^+_j(t) = 0\}$, resp. $\{t \in [t_I,t_F] : z^-_j(t) = 0\}$, for $j=1,\ldots,m$.

Note that formulation (\ref{formulation-classique}) of the dual problem dates back to the work of Neustadt \cite{neustadt} and was preferred to the primal formulation for numerical solution of the optimal control problem. One objective of this paper is to show that the recent advances on the moment problem give an efficient computation procedure for the primal problem. In particular, one obtains very good approximations of an optimal solution of (\ref{plp}) (impulse times and amplitudes).
\begin{remark}
The vector $G(t)y$ involved in LP problem (\ref{dlp})
is known as the primer vector introduced in the seminal
work \cite{Lawden63}. This primer vector is defined as the velocity
adjoint vector arising by the application of Pontryagin's maximum principle
to optimal trajectory problems. The primer vector must satisfy Lawden's
well-known necessary conditions for an optimal impulsive trajectory.
\end{remark}

\section{Integration of the LTV ODE}\label{integration}

In order to compute matrix $F(t)$, the last step before obtaining a tractable problem, we have to integrate numerically
the ODE $\dot{x}(t)=A(t)x(t)$.
Matrix $G(t)$ solves numerically the LTV system of equations $F(t)G^T(t)=B(t)$.
In practice, we use the Matlab software package Chebfun \cite{chebfun}
to build Chebyshev polynomial interpolants of the problem data $A(t)$ and $B(t)$.
Some of its specialized numerical routines to compute $F(t)$
and $G(t)$ have been used here.

We want to characterize the error of approximating $G(t)$ by its Chebyshev
polynomial interpolant on $d$ points, see \cite{trefethen} for an introduction
on this subject. Define $G_d(t)$ as the Chebyshev interpolant of $G(t)$ on $d$
points. Assume furthermore that the error induced by algebraic manipulations
for obtaining $G_d(t)$ can be properly controlled. Then the approximation error
$e_d:= \|G_d-G\|_{\infty}$ for large $d$ is conditionned by the regularity of $G(t)$.
We recall the main results of approximation theory which can all be found
in \cite{trefethen}:
\begin{itemize}
\item if $G(t)$ is $k$-times continuously differentiable, then $e_d \rightarrow 0$
at the algebraic rate of $\mathcal{O}(d^{-k})$ as $d \rightarrow \infty$;
\item if $G(t)$ is $k$-times differentiable with its $k$-th derivative of bounded variation,
then the algebraic rate is $\mathcal{O}(d^{-k})$;
\item if $G(t)$ is analytic, then there exists a positive $\rho$ such that
the algebraic rate is $\mathcal{O}(\rho^{-d})$.
\end{itemize}
This imposes some minimal requirements for $A(t)$ and $B(t)$ for our numerical
approach to work. Indeed, assuming $A(t)$ of bounded variation and $B(t)$ differentiable with derivative of bounded variation guarantees that the error converges
to zero for large approximation orders.

We now show that if $e_d \rightarrow 0$ as $d \rightarrow \infty$, we can build a hierarchy of moment relaxations of (\ref{m}) involving only approximate data and converging to $p^*$. For this,
let $F_d(t)$ denote the Chebyshev interpolant of $F(t)$ on $d$ points,
and let $h_d := F^{-1}_d(t_F)x(t_F)-F^{-1}_d(t_I)x(t_I)$.

\begin{lemma}
Consider the following relaxed moment problem with approximate data:
\begin{equation} \label{eq:probMomApprox} 
\begin{array}{rcl}
\tilde{p}_d^* = & \min & \|\mu\|_{TV}\\
& \mathrm{s.t.} &  \left| \int \! G_d^T \, d\mu - h_d  \right| \leq e_d \,( \|B\|\,|x(t_F)|
+ \|\mu\|_{TV}).
\end{array}
\end{equation}
Then $\tilde{p}_d^* \uparrow p^*$ as $d \rightarrow \infty$
\end{lemma}
\begin{proof}
By the same argument as Lemma \ref{relax}, a solution for (\ref{eq:probMomApprox}) is attained.
Furthermore, any solution $\mu^*$ of (\ref{m}) is feasible for (\ref{eq:probMomApprox}), as
\[
\begin{array}{rcl}
\left| \int \! G_d^T \, d\mu^* - h_d  \right| & \leq & \left| \int \! G^T \, d\mu^* - h  \right|  +    \left| \int \! (G_d^T-G^T) \, d\mu^*  \right| +  ||B||\,|x(t_F)| \, \left| h_d - h  \right| \\
&\leq & 0 + e_d \, ||\mu^*||_{\textrm{TV}} + e_d\,||B||\,|x(t_F)|.
\end{array}
\]
Therefore, $\tilde{p}_e^* \leq p^*$ holds.

For the convergence, consider the auxiliary relaxed problem with the true data instead:
\begin{equation} \label{eq:probMomRelax}
\begin{array}{rcl}
p_d^* = & \min & \|\mu\|_{TV} \\
& \mathrm{s.t.} & \left| \int \! G^T \, d\mu - h  \right| \leq 2 e_d \,(\|B\|\,|x(t_F)| +  \|\mu\|_{TV}).
\end{array}
\end{equation}
By similar arguments, we have that $p_d^* \leq \tilde{p}_d^* \leq p^*$. Because of this uniform bound on the minimizers of (\ref{eq:probMomRelax}), standard arguments (see for instance the second part of the Banach-Saks-Steinhaus theorem in \cite[Section 13.5]{royden}) show that $p^* \leq \liminf p_d^*$, which show indeed that $p_d^* \rightarrow p^*$, hence $\tilde{p}_d^* \rightarrow p^*$.
\end{proof}

\section{Solving the LP on measures}

To summarize, we have formulated our LTV optimal control problem
as a primal-dual LP pair (\ref{plp}-\ref{dlp}). The coefficients
in these LP problems are entries of matrix function $G(t)$
and vector $h$. These data are calculated by numerical
integration, using Chebyshev polynomial approximations,
as explained in the last section. For a given degree $d$
of the polynomial approximation, LP problem (\ref{plp})
is a particular instance of a generalized problem of
moments. It can be seen as an extension of the approach of \cite{sicon}
which was originally designed for classical optimal control problems
with polynomial dynamics. Alternatively, it can also be understood
as an application of the approach of \cite{impulse}, but after
integration of the ODE, which is here possible because of linearity
of the dynamics in the state and control.

An infinite-dimensional LP on measures can be solved approximately by
a hierarchy of finite-dimensional linear matrix inequality (LMI) problems,
see \cite{sicon,impulse,roa} for details
(not reproduced here). The main idea behind the hierarchy is to manipulate
each measure via its moments truncated to degree $d$.
Note that here the measures are univariate (depending on time only). Therefore, for a problem involving polynomials up to degree $d$, the $d$-th LMI condition is necessary and sufficient. The hierarchy presented in this paper comes thus from the polynomial approximation of the continuous data, not from the moment truncation.
In contrast, when dealing with multivariate measures as in \cite{sicon,impulse,roa},
we use a hierarchy of necessary LMI conditions which become sufficient
only asymptotically.

Generally speaking, the number of variables in an LMI relaxation of order $d$
of a multivariate measure LP grows polynomially in $d$, but the exponent is the number of
variables entering the measures. If the number of variables is equal to 5 or more,
the growth is fast, and
only LMI relaxation of small orders can be solved at a reasonable computational cost.
It is therefore crucial to reduce as much as possible the number of variables
entering the measures, so as to reduce the overall computational burden.
One contribution of our paper is precisely to show that for LTV optimal
control problems, we can manipulate measures of the time variable only.
This allows for LMI relaxations of large order to be solved, with
$d$ a few hundreds. 

\section{Examples}

\subsection{Scalar polynomial example}

Consider the optimal control problem (\ref{ocp}):
\[
\begin{array}{rcll}
p^* & = & \inf & \displaystyle \|u\|_1 := \int_0^1 |u(t)|dt \\
& & \mathrm{s.t.} & \dot{x}(t) = t(1-t)u(t)\\
& & & x(0) = 0, \quad x(1) = 1
\end{array}
\]
where the minimization is w.r.t. a function $u \in L^1([0,1])$.
Since $L^1([0,1])$ is not the dual of any normed space,
the problem must be relaxed to the optimal control problem (\ref{p}):
\[
\begin{array}{rcll}
p^* & = & \min & \displaystyle \|\mu\|_{TV} := \int_0^1 |\mu|(dt) \\
& & \mathrm{s.t.} & x(dt) = t(1-t)\mu(dt)\\
& & & x(0) = 0, \quad x(1) = 1
\end{array}
\]
where the minimization is now w.r.t. a measure $\mu \in M([0,1])$.
Following the approach of Section \ref{moments}, we readily obtain
$F(t)=1$, $G(t)=t(1-t)$ and hence the moment problem (\ref{m}):
\[
\begin{array}{rcll}
p^* & = & \min & \|\mu\|_{TV} \\
& & \mathrm{s.t.} & \displaystyle \langle t(1-t),\mu \rangle := \int_0^1 t(1-t)\mu(dt)  = 1.
\end{array}
\]
Decomposing $\mu=\mu^+-\mu^-$ as a difference of two nonnegative
measures we obtain the LP (\ref{plp}):
\[
\begin{array}{rcll}
p^* & = & \inf & \langle 1,\mu^+ \rangle + \langle 1, \mu^- \rangle  \\
& & \mathrm{s.t.} &
\displaystyle \langle t(1-t),\mu^+ \rangle  - \langle t(1-t),\mu^- \rangle = 1 \\
& & & \mu^+ \geq 0, \quad \mu^- \geq 0
\end{array}
\]
where the minimization is w.r.t. measures $\mu^+, \mu^-$,
and the LP (\ref{dlp}):
\[
\begin{array}{rcll}
d^* & = & \sup & y  \\
& & \mathrm{s.t.} & -1 \leq t(1-t)y \leq 1, \quad \forall~t \in [0,1]
\end{array}
\]
where the maximization is w.r.t. $y \in {\mathbb R}$.
For this latter problem, we readily obtain that the maximum
$d^*=4$ is attained for the choice $y=4$. Since polynomial
$4t(1-t)$ attains its maximum at $t=\frac{1}{2}$ it follows
from the discussion after Lemma \ref{nogap}
that the optimal measure solution is $\mu^+=4\delta_{\frac{1}{2}}$
and $\mu^-=0$, achieving $p^*=4$.

As this problem has polynomial data, we can
readily find the optimal solution, with the following GloptiPoly 3 script \cite{gloptipoly3}, instead of using the approximation techniques of Section V:
\begin{verbatim}
>> mpol tp tn;
>> mp = meas(tp); % measure \mu^+
>> mn = meas(tn); % measure \mu^-
>> P = msdp(min(mass(mp)+mass(mn)), ...
   mom(tp*(1-tp))-mom(tn*(1-tn)) == 1, ...
   tp*(1-tp) >= 0 ... % support of \mu^+
   tn*(1-tn) >= 0);   % support of \mu^-
>> [stat, obj] = msol(P);
obj =
    4.0000
>> double(mmat(mp)) % moment mat. of \mu^+
ans =
    4.0000    2.0001
    2.0001    1.0001
>> double(mmat(mn)) % moment mat. of \mu^-
ans =
   1.0e-08 *
    0.2673    0.2036
    0.2036    0.2186
\end{verbatim}

\subsection{Nonsmooth trajectories}\label{Section:numerical_B}

This example is meant to illustrate the numerical difficulty we
may face when integrating LTV ODEs with discontinuous state matrix
$A(t)$. Consider the following optimal control problem (\ref{ocp}):
\[
\begin{array}{rcll}
p^* & = & \inf & \|u\|_1 \\
& & \mathrm{s.t.} & \dot{x}(t) = \mathrm{sign}(t)x(t)+u(t) \\
& & & x(-1) = -1, \quad x(1) = 1
\end{array}
\]
where the minimization is w.r.t. a function $u \in L^1([-1,1])$.
Upon solving the ODE $\dot{F}(t) = \mathrm{sign}(t)F(t)$, $F(-1)=1$ we obtain
$F(t)=e^{-1+|t|}$, $G(t)=e^{1-|t|}$
and hence the moment problem (\ref{m}):
\[
\begin{array}{rcll}
p^* & = & \inf & \|\mu\|_{TV} \\
& & \mathrm{s.t.} & \langle e^{1-|t|},\mu \rangle = 2
\end{array}
\]
where the minimization is now w.r.t. a signed measure $\mu \in M([-1,1])$.
The dual problem (\ref{dlp}) reads:
\[
\begin{array}{rcll}
d^* & = & \sup & 2y \\
& & \mathrm{s.t.} & -1 \leq ye^{1-|t|} \leq 1 \quad \forall~t\in[-1,1]\\
\end{array}
\]
where the maximization is w.r.t. a real scalar $y$.
The optimal dual solution can easily be found analytically as $y^*=e^{-1}$, from which it follows that
the optimal primal solution is $\mu^*=2e^{-1}\delta_{t=0}$.

\begin{table}[h!]
\centering
\begin{tabular}{c|c|c}
$d$ & $\|G-G_d\|_{\infty}$ & $|p^*-p_d|/|p^*|$ \\ \hline
10 & 0.1995 & 0.1243 \\
20 & 0.0899 & 0.0555 \\
30 & 0.0579 & 0.0357 \\
40 & 0.0427 & 0.0263 \\
50 & 0.0338 & 0.0208 \\
60 & 0.0280 & 0.0172 \\
70 & 0.0239 & 0.0147 \\
80 & 0.0208 & 0.0128 \\
90 & 0.0184 & 0.0114 \\
100 & 0.0166 & 0.0102 \\
\end{tabular}
\caption{Approximation errors vs. polynomial approximation orders for
discontinuous dynamics.\label{ex2tab}}
\end{table}

If we want to apply the numerical integration approach of Section \ref{integration},
discontinuity of $A(t)=\mathrm{sign}(t)$ turns out to be a problem. Indeed, we can check
that $G(t)$ presents a cusp at $t=0$. That is, for an even number $d$ of
Chebyshev interpolation points, the maximum interpolation error
$\|G-G_d\|_{\infty}$ is located
where the optimal impulse should be, while for an odd $d$ the interpolant
is exact at the cusp. In Table \ref{ex2tab} we present the approximation
errors on function $G$ and on the optimal cost $p_d^*$ as functions of $d$, found by our numerical method.
The results exhibit the expected linear decrease in the approximation error,
resulting in a linear decrease in the optimal cost. The error for moderate
orders $(d\approx 100)$ is acceptable though far from excellent.

\begin{table}[h!]
\centering
\begin{tabular}{c|c|c}
$d$ & $\|G-G_d\|_{\infty}$ & $|p^*-p_d|/|p^*|$ \\ \hline
2 & $2.1187\cdot10^{-1}$ & $3.4307\cdot10^{-3}$ \\
4 & $1.0851\cdot10^{-3}$ & $1.0406\cdot10^{-5}$ \\
6 & $2.2540\cdot10^{-6}$ & $1.5281\cdot10^{-8}$ \\
8 & $2.5115\cdot10^{-9}$ & $1.3166\cdot10^{-11}$ \\
10 & $1.7422\cdot10^{-12}$ & $5.9438\cdot10^{-13}$ \\
12 & $1.1336\cdot10^{-15}$ & $5.8093\cdot10^{-12}$ \\
14 & $<10^{-15}$ & $1.3729\cdot10^{-11}$ \\
16 & $<10^{-15}$ & $3.1846\cdot10^{-12}$ \\
18 & $<10^{-15}$ & $2.5462\cdot10^{-10}$ \\
\end{tabular}
\caption{Approximation errors vs. polynomial approximation orders for
piecewise analytic dynamics.\label{ex2btab}}
\end{table}

Now, if we split the time interval around the cusp, the restrictions
of $G(t)$ on the intervals $[-1,0]$ and $[0,1]$ are respectively $G_L(t):=e^{1+t}$ and
$G_R(t):=e^{1-t}$ which are both analytic. The updated moment problem
with piecewise analytic data reads:
\[
\begin{array}{rcll}
p^* & = & \inf & \|\mu_L\|_{TV}+\|\mu_R\|_{TV} \\
& & \mathrm{s.t.} & \displaystyle \int_{-1}^0 G_L(t)\mu_L(dt)
+ \int_{0}^{1} G_R(t)\mu_R(dt) = 2.
\end{array}
\]
In Table \ref{ex2btab}, we present the updated polynomial approximation
errors, which show the expected exponential decrease. For the cost,
the decrease is exponential until a relative error of about $10^{-10}$
after which the error comes from the semidefinite solver (used to
solve the LMI hierarchy of the moment problem), not from
the polynomial approximation.

\subsection{A fuel-optimal linear impulsive guidance rendezvous problem for an elliptic reference orbit}\label{Section:numerical_rdv}

\begin{figure*}[!t]
\tiny
\begin{equation}
\begin{array}{@{}r@{\:}c@{\:}l@{}}
A(t) & = & \left[\begin{array}{cccc}
0 & 0 & 1 & 0\\
0 & 0 & 0 & 1\\
\bar{n}^2 e\,\cos \nu(t) \left(\dfrac{1 + e\,\cos \nu(t)}{1 - e^2} \right)^3  &
-2\,\bar{n}^2\,e\,\sin \nu(t) \left(\dfrac{1 + e\,\cos \nu(t)}{1 - e^2} \right)^3 &
0 &
2\bar{n}\dfrac{\left(1 + e\,\cos \nu(t)\right)^2}{\left(1 - e^2\right)^{3/2}}\\
2\,\bar{n}^2\,e\,\sin \nu(t) \left(\dfrac{1 + e\,\cos \nu(t)}{1 -e^2} \right)^3 &
\bar{n}^2(3+e\cos\nu(t))\left (\dfrac{1+e\cos\nu(t)}{1-e^2}\right )^3 &
-2\bar{n}\dfrac{\left(1 + e\,\cos \nu(t)\right)^2}{\left(1 - e^2\right)^{3/2}} & 0
\end{array}\right ]\\
B(t) & = & \left[\begin{array}{cc}
0 & 0\\
0 & 0\\
\bar{n}^2 & 0\\
0 & \bar{n}^2
\end{array}\right]\end{array}
\label{A_TH_B_TH}
\end{equation}
\hrulefill
\end{figure*}

Finally, an illustration based on a realistic case of a far range rendezvous in a linearized gravitational field is given. The general framework of the minimum-fuel fixed-time coplanar rendezvous problem in a linear setting is recalled in \cite{Arzelier11}, where an indirect method based on primer vector theory is proposed. Under Keplerian assumptions and for an elliptic reference orbit, the complete rendezvous problem may be decoupled between the out-of-plane rendezvous problem for which an analytical solution may be found and the coplanar problem. Therefore, only a coplanar elliptic rendezvous problem based on the Tschauner-Hempel equations \cite{Tschauner65} is studied thereafter. The associated optimal control problem
(\ref{ocp}) has a 4-dimensional ($n=4$) state vector composed of the relative positions
(denoted here $x(t)$, $z(t)$) and respective velocities in the LVLH frame \cite{Louembet11}
and a 2-dimensional ($m=2$) control vector $u(t)$
(one control in the $x$-direction, one control in the $z$-direction).
The state space matrices $A(t)$ and $B(t)$ are given in equ. (\ref{A_TH_B_TH}),
with $\bar{n}=34.0094$ is the mean angular motion, $e=4.0000\cdot~10^{-3}$ is the
eccentricity, $\nu(t)$ is the true anomaly
of the reference orbit satisfying for all $t\geq 0$ the Kepler equation:
\[
\begin{array}{rcl}
\bar{n}t & = & E(t)-e\,\sin E(t)\\
\tan \dfrac{\nu(t)}{2} & = & \left (\dfrac{1+e}{1-e}\right )^{1/2}\tan\dfrac{E(t)}{2}
\end{array}
\label{Kepler}
\]
where $E(t)$ is the eccentric anomaly.

For this problem, $G(t)$ can be approximated below the $10^{-8}$ resolution of the SDP solver by polynomials of degree $100$ on the given time interval. This implies that the problem can be solved numerically by LMI relaxations of order $50$, with a computational load of a few seconds. The assembly of $G(t)$ using Chebfun routines take a few seconds as well. All of this was done without any sort of problem-specific optimization, and with standard Matlab code.

To solve the optimal control problem, direct methods based on the solution of a linear programming (LP) problem can be used as in \cite{Mueller08,Louembet11}. For an a priori fixed number of impulsive
maneuvers at given times, an LP problem is formulated and solved numerically. Its solution is therefore suboptimal, depending strongly upon the number of impulsions. This makes it hard to evaluate
how far the solution could be from the global optimum.

The particular instance, studied in this paper, is borrowed from the PRISMA test bed and GNC experiments from \cite{Berges11}. PRISMA programme
is a cooperative effort between the Swedish National Space Board (SNSB),
the French Centre National d'Etudes Spatiales (CNES), the German Deutsche
Zentrum f\"{u}r Luft- und Raumfahrt (DLR) and the Danish Danmarks Tekniske Universitet (DTU) \cite{Larsson06}. Launched on June 15, 2010 from Yasny (Russia), it was intended to test in-orbit
new guidance schemes (particularly autonomous orbit control) for formation flying and rendezvous
technologies. This mission includes the FFIORD experiment led by CNES, which features a rendezvous maneuver (formation acquisition).
To save fuel and allow for in-flight testing throughout the FFIORD experiment, the rendezvous maneuver must last several orbits. Duration of the
rendezvous is approximately 14.25 orbital periods, each of duration $5920s$,
i.e. $t_I=0$ and $t_F=84360$ in problem (\ref{ocp}).

\begin{table}[h!]
\center
\begin{tabular}{c|c|c|c|c}
  & 2-impulse & LP (20 i.) & LP (2000 i.) & LMI\\
\hline $t_0$ & 0 & 0 & 0 & 0\\
\hline $t_1$ & - & $8880$ & 2321 & 2140\\
\hline $t_2$ & - & $75480$ & 2363 & 82350\\
\hline $t_3$ & 84360 & $84360$ & 82334 & -\\
\hline
$u(t_0)$ & $\begin{array}{c}0.02\\ -0.1985\end{array}$ &$\begin{array}{c} 0.0006 \\ 0\end{array} $ & $\begin{array}{c}0.0174\\ 0\end{array}$ & $\begin{array}{c}0.0172\\ 0\end{array}$ \\
\hline
$u(t_1)$ & - &$\begin{array}{c} 0.0014 \\   0 \end{array} $ & $\begin{array}{c}0.0193\\ 0\end{array}$ & $\begin{array}{c}0.0011\\ 0\end{array}$\\
\hline
$u(t_2)$ & - &$\begin{array}{c} -0.0146 \\   0 \end{array} $ & $\begin{array}{c}-0.0003\\ 0\end{array}$ &$\begin{array}{c}-0.019\\ 0\end{array}$\\
\hline
$u(t_3)$ & $\begin{array}{c}-0.0215 \\ 0.218\end{array} $ &$\begin{array}{c}-0.0207  \\  0\end{array} $ & $\begin{array}{c}-0.019\\ 0\end{array}$ & -\\
\hline
cost m/s & $0.4588$ & $0.04072$ & 0.03736 & 0.03736
\end{tabular}
\caption{\label{prisma_res}  Comparisons of results of moment LMI approach, two-impulse solution and direct LP solution for 200 and 2000 impulses for the PRISMA case study}
\end{table}

The moment LMI approach is compared to the classical 2-impulse solution and to the direct approach with 20 and 2000 pre-assigned evenly distributed impulses. These results are presented in Table \ref{prisma_res},
where $t_k$ are the impulsion times and $u(t_k)$ the impulsion directions, for $k=0,1,2,3$.
It is interesting to note that the 2-impulse solution has a prohibitive cost when compared to the optimal solution found by the direct method and the moment LMI approach. This is mainly due to the extra thrusts in the $z$-direction that are not necessary to realize this particular rendezvous. Indeed, a remarkable feature of the optimal solution is that it exhibits a terminal coast, unlike the 2-impulse and direct 20-impulse solution. Note also that the 2-impulse solution leads to a very different trajectory in the orbital plane as shown by Figure \ref{fig:trajectoires_comparaisons}. The optimal fuel-cost computed via the moment LMI approach is less than half the one obtained in \cite{Berges11}
via a pseudo-closed loop technique ($0.086$ m/s).
\begin{figure}[ht!]
\centering
\includegraphics[width=\columnwidth]{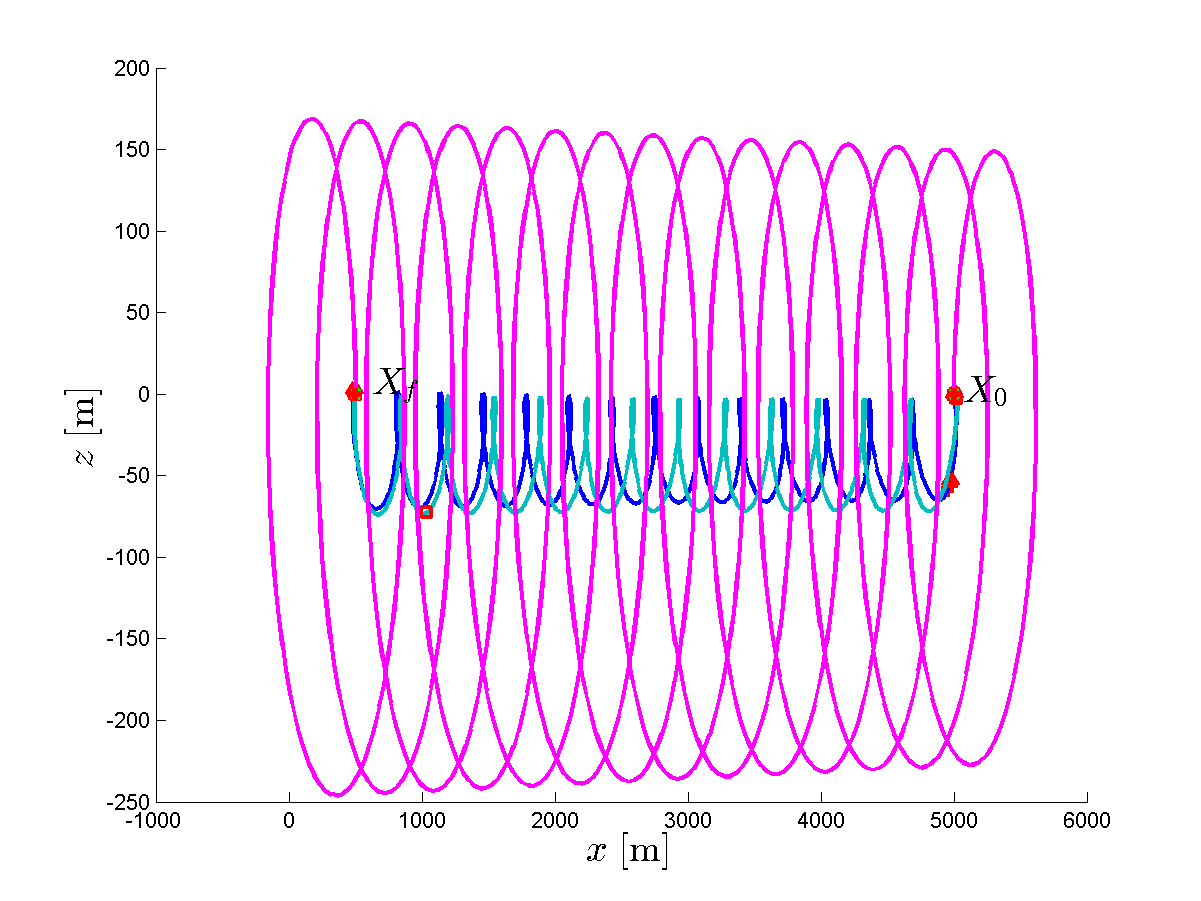}
\caption{Trajectories in the orbital plane: 2-impulse solution (pink), optimal solution (blue), Direct 20-impulses solution (cyan).}
\label{fig:trajectoires_comparaisons}
\end{figure}
Figure \ref{fig:trajectoires_comparaisons_20_2000_moments} shows the importance of the pre-assigned number of impulses for the direct method to recover the optimal solution. Indeed, for a small number of impulses, the solution given by the direct method is a crude approximation of the optimal solution obtained by the moment LMI approach and by the direct approach with 2000 pre-assigned impulses.
The impulse locations are indicated on the trajectories (red triangles for the optimal solution and red squares for the 20-impulse solution). The differences of locations of the impulses for the four methods used on the PRISMA example are depicted on Figure \ref{fig:impulsions_comparaisons}.
\begin{figure}[ht!]
\centering
\includegraphics[width=\columnwidth]{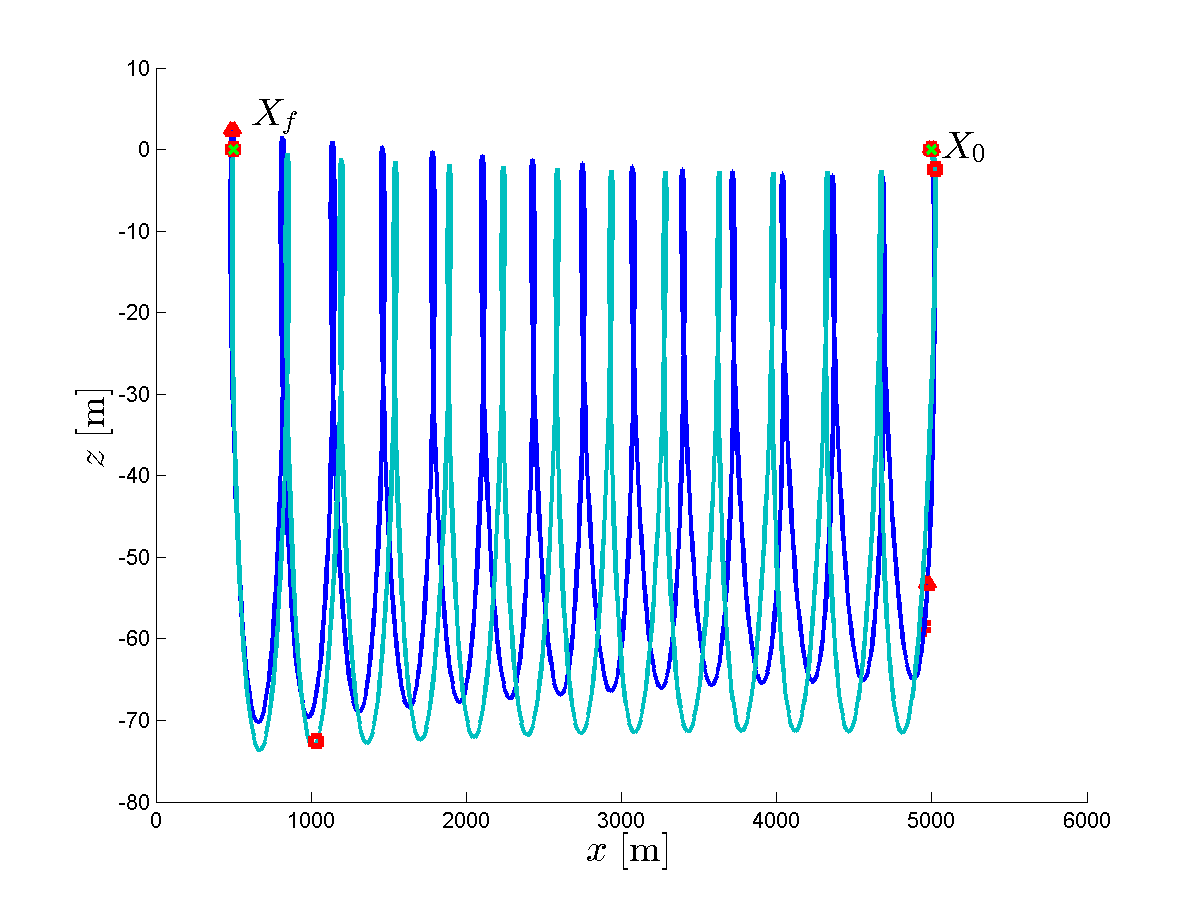}
\caption{Trajectories in the orbital plane: optimal solution (blue), Direct 20-impulses solution (cyan) .}
\label{fig:trajectoires_comparaisons_20_2000_moments}
\end{figure}

\begin{figure}[ht!]
\centering
\includegraphics[width=\columnwidth]{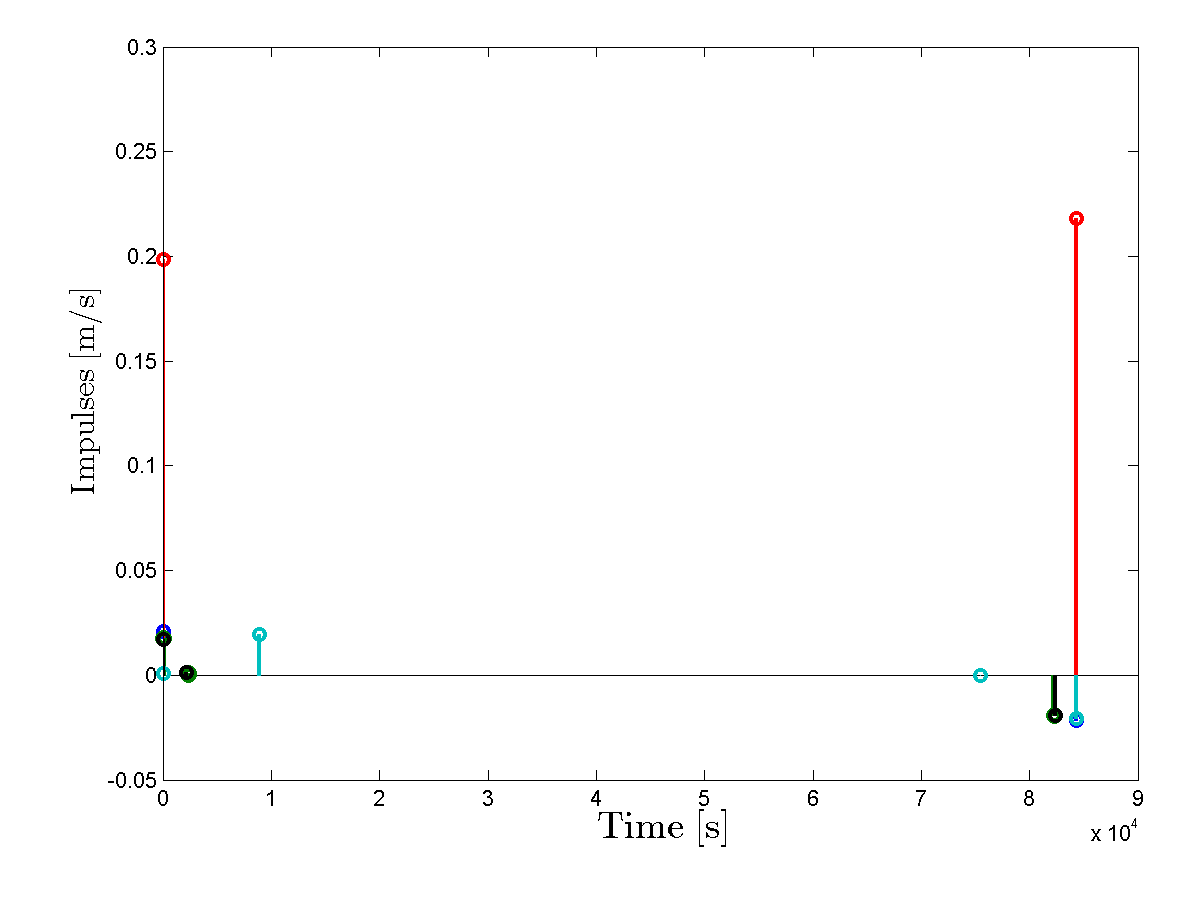}
\caption{Impulses for 2-impulse solution (red and blue), direct 20-impulse solution (cyan), optimal solution (black).}
\label{fig:impulsions_comparaisons}
\end{figure}

\section{Conclusion}

In this paper, we revisit classic impulsive control theory from the 1960s with modern numerical tools stemming from convex programming and approximation theory. The theory leads to the reformulation of a control problem as a one-dimensional problem of moments. The numerical tools allow for its efficient numerical solution without any expert knowledge needed.

Several extensions of our work are possible and will be included in an extended version of this paper. First of all, the class of dynamics can easily be extended to cover those with a forced drift term, i.e. to dynamics
\begin{equation*}
  \dot{x} = A(t)x(t) + B(t)u(t) + w(t)
\end{equation*}
with $w(t) \in C^1([t_I, t_F],\mathbb{R}^n)$. The second main extension is to enforce positivity constraints on linear combinations of the states using additional positive measures. Finally, the method can also consider other norms of the form $||u(\cdot)||_p := \int \! |u(t)|_p \, dt$ with $p=2,3,...$, instead of the $p=1$ case that is the focus of this paper. This can be done since norms with integer
exponents are semidefinite representable.

For the specific case of orbital rendezvous as exemplified in Section~\ref{Section:numerical_rdv}, one can exploit the specific symmetries of (\ref{A_TH_B_TH}) to approximate  $G(t)$ off-line on half an orbital period, and partition the time interval accordingly such as presented in Ex.~\ref{Section:numerical_B}. This will lead to accurate, low order polynomial approximations, and one could expect to obtain computational loads compatible for an online implementation of a model predictive control loop.

\end{document}